\documentclass{article}
\usepackage[utf8]{inputenc}

\usepackage[paper=a4paper,top=1in, bottom=1in, right=1.5in, left=1.2in]{geometry}
\usepackage{amsthm, amssymb, amsfonts, amsmath}
\usepackage{graphicx}
\usepackage{subfigure}
\usepackage{tikz}
\usepackage{enumitem}
\usepackage{mathtools}
\usepackage{mathrsfs}
\usepackage{tikz-cd}
\usepackage{hyperref, mathabx}
\usepackage{color}
\numberwithin{equation}{section}
\usepackage{ragged2e}

\newtheorem{theorem}{Theorem}[section] 
 
\newtheorem{lemma}{Lemma}[section] 

\newtheorem{remark}{Remark}[section]

\title{
Travelling Wave Solutions of the Bistable Reaction-Diffusion Equation with Nonlinear jump discontinuity }
\author{Hou, Lingling ({\color{blue}hll67800@163.com}) }
\date{ June, 2022  }

\begin{document}

\maketitle
\noindent \textbf{Abstract.} \ 
This paper is concerned with the existence and the stability of travelling wave solution to a bistable reaction-diffusion equation with a jump discontinuious point on nonlinear term. Sub-super solution method is used throughout this paper. As a matter of fact, this nonlinear discontinuity term is frequently emerging when study the travelling wave solution of reaction-diffusion-ODE systems. Hence, it is meaningful to consider this kind of problem.

\section{Problem}

Consider the reaction-diffusion equation:
\begin{equation}\label{eq-RDE}
    u_t=u_{xx}+f(u),
\end{equation}
where
\begin{equation*}
    f(u)=\left\{\begin{aligned}
    f_0(u), \quad& u\in[0,a] \\ f_1(u), \quad& u\in[a,1]
    \end{aligned}\right.
\end{equation*}
with $0<a<1$. Here $f_0\in C^1[0,a]$ and $f_1\in C^1[a,1]$ satisfy
\begin{itemize}
    \item [(H1)] $f_0(0)=0, f'_0(0)<0,f_1(0)=0$ and $f'_1(1)<0$.
    \item [(H2)] $f_0(u)<0$ for $u\in(0,a]$, $f_1(u)>0$ for $u\in[a,1)$.
    \item [(H3)] $\int_0^a f_0(u)du+\int_a^1 f_1(u)du>0$. 
\end{itemize}

I would like to study the travelling wave solution to \eqref{eq-RDE} under the moving coordinate $z=x+ct$, satisfies $u(z)\to0,$ as $z\to-\infty$, and $u(z)\to1,$ as $z\to+\infty$. The method I am going to adopt is basically rely on the works of \cite{Chen}. 

First, denote
\[\underline{\alpha}=\inf\limits_{u\in[0,a]}\frac{f_0(u)}{u},\quad \overline{\alpha}=\sup\limits_{u\in[0,a]}\frac{f_0(u)}{u},\] and
\[\underline{\beta}=\inf\limits_{u\in[a,1]}\frac{f_1(u)}{u-1},\quad \overline{\beta}=\sup\limits_{u\in[a,1]}\frac{f_1(u)}{u-1}.\]
Due to the assumption {\rm (H1)-(H2)}, and notice $-\infty<\underline{\alpha}\leq\bar{\alpha}<0,  -\infty<\underline{\beta}\leq\bar{\beta}<0$.

\begin{remark}\label{rmk:fu}
Thanks to the assumption {\rm (H3)}, it holds the fact:
\begin{itemize}
    \item {\rm(1)} For sufficiently small positive $\epsilon$,  it holds $f_0(u)-\epsilon u<\underline{\alpha}u<f_0(u)$ and  $f_0(u)<\overline{\alpha}u<f_0(u)+\epsilon u$, when $0\leq u\leq a$. On the other hand, $f_1(u)<\underline{\beta}(u-1)<f_1(u)-\epsilon(u-1)$ and $f_1(u)+\epsilon(u-1)<\overline{\beta}(u-1)<f_1(u)$, when $a\leq u\leq 1$.
    \item {\rm(2)} $\sqrt{-\overline{\alpha}}a<\sqrt{-\underline{\alpha}}a<\sqrt{-\overline{\beta}}(1-a)<\sqrt{-\underline{\beta}}(1-a)$.
\end{itemize}
\end{remark}

Next, define four discontinuous functions:
\begin{equation}
    \underline{f}(u)=\left\{\begin{aligned}
   & \underline{\alpha}u,\quad & u\in[0,a] \\ &  \underline{\beta}(u-1),\quad & u\in[a,1]
    \end{aligned}
    \right.
\end{equation}
\begin{equation}
    \overline{f}(u)=\left\{\begin{aligned}
   & \overline{\alpha}u,\quad & u\in[0,a] \\ &  \overline{\beta}(u-1),\quad & u\in[a,1]
    \end{aligned} 
    \right.
\end{equation}
\begin{equation}
    \underline{g}(u)=\left\{\begin{aligned}
   & \underline{\alpha}u,\quad & u\in[0,a] \\ &  \overline{\beta}(u-1),\quad & u\in[a,1]
    \end{aligned}
    \right.
\end{equation}
\begin{equation}
    \overline{g}(u)=\left\{\begin{aligned}
   & \overline{\alpha}u,\quad & u\in[0,a] \\ &  \underline{\beta}(u-1),\quad & u\in[a,1]
    \end{aligned} 
    \right.
\end{equation}
Thus it's easily to find the $C^0$ solutions of 
\begin{equation}\label{eq:TWsol}
    \left\{\begin{aligned}
    & u_{zz}-cu_z+h(u)=0, \\ & u(0)=a, u(-\infty)=0, u(+\infty)=1
    \end{aligned}\right.
\end{equation}
for $h(u)=\underline{f}(u), h(u)=\overline{f}(u), h(u)=\underline{g}(u)$ and $h(u)=\overline{g}(u)$.
In order to consider the $C^1$ solution of \eqref{eq:TWsol}, the following four cases are needed.
\begin{itemize}
    \item (i) $h(u)=\underline{f}(u)$. The $C^0$ solutions for any $c\geq0$ is given:
    \begin{equation}\label{eq:under-u}
        \underline{u}(z)=\left\{ \begin{aligned}
        &\quad a\exp({\underline{\lambda}_0^+(c)z}), \quad &z<0, \\ & 1+(a-1)\exp({\underline{\lambda}_1^-(c)z}), \quad &z\geq 0. 
        \end{aligned} \right.
    \end{equation}
    Where $\underline{\lambda}_0^+(c)=(c+\sqrt{c^2-4\underline{\alpha}})/2>0$, $\underline{\lambda}_1^-(c)=(c-\sqrt{c^2-4\underline{\beta}})/2<0$. By choosing $\underline{c}$ such that 
    \[a=\frac{\underline{\lambda}_1^-(\underline{c})}{\underline{\lambda}_1^-(\underline{c})-\underline{\lambda}_0^+(\underline{c})},\]
    then we get a unique $C^1$ solution of \eqref{eq:TWsol}.
    \item (ii) $h(u)=\overline{f}(u)$. Similarly to the case (i), the $C^0$ solutions for any $c\geq0$ is given:
    \begin{equation}\label{eq:over-u}
        \overline{u}(z)=\left\{ \begin{aligned}
        &\quad a\exp({\overline{\lambda}_0^+(c)z}), \quad &z<0, \\ & 1+(a-1)\exp({\overline{\lambda}_1^-(c)z}), \quad &z\geq 0. 
        \end{aligned} \right.
    \end{equation}
    Where $\overline{\lambda}_0^+(c)=(c+\sqrt{c^2-4\overline{\alpha}})/2>0$, $\overline{\lambda}_1^-(c)=(c-\sqrt{c^2-4\overline{\beta}})/2<0$. By choosing $\overline{c}$ such that 
    \[a=\frac{\overline{\lambda}_1^-(\overline{c})}{\overline{\lambda}_1^-(\overline{c})-\overline{\lambda}_0^+(\overline{c})},\]
    then we get a unique $C^1$ solution of \eqref{eq:TWsol}.
    \item (iii) $h(u)=\underline{g}(u)$. The $C^0$ solutions for any $c\geq0$ is given:
    \begin{equation}\label{eq:check-u}
        \check{u}(z)=\left\{ \begin{aligned}
        &\quad a\exp({\underline{\lambda}_0^+(c)z}), \quad &z<0, \\ & 1+(a-1)\exp({\overline{\lambda}_1^-(c)z}), \quad &z\geq 0. 
        \end{aligned} \right.
    \end{equation}
    By choosing $\check{c}$ such that 
    \[a=\frac{\overline{\lambda}_1^-(\check{c})}{\overline{\lambda}_1^-(\check{c})-\underline{\lambda}_0^+(\check{c})},\]
    then we get a unique $C^1$ solution of \eqref{eq:TWsol}.
    \item (iv) $h(u)=\overline{g}(u)$ The $C^0$ solutions for any $c\geq0$ is given:
    \begin{equation}\label{eq:hat-u}
        \hat{u}(z)=\left\{ \begin{aligned}
        &\quad a\exp({\overline{\lambda}_0^+(c)z}), \quad &z<0, \\ & 1+(a-1)\exp({\underline{\lambda}_1^-(c)z}), \quad &z\geq 0. 
        \end{aligned} \right.
    \end{equation}
   By choosing $\hat{c}$ such that 
    \[a=\frac{\underline{\lambda}_1^-(\hat{c})}{\underline{\lambda}_1^-(\hat{c})-\overline{\lambda}_0^+(\hat{c})},\]
    then we get a unique $C^1$ solution of \eqref{eq:TWsol}.
\end{itemize}

\section{Existence of travelling wave solution}
For the existence and uniqueness of the travelling wave solution to \eqref{eq-RDE}, we need to prove a few lemmas.

Let $w^-(z)=u_z(z), \mbox{for}~z<0$ and $w^+(z)=u_z(z), \mbox{for}~z>0$. Also, $\underline{w}(z)=\underline{u}_z(z)$ and $\overline{w}(z)=\overline{u}_z(z)$.

\begin{lemma}\label{lem:subp-sol}
Assume that {\rm(H1)-(H3)} hold. Then $\underline{u}^-(z)=a\exp({\underline{\lambda}_0^+(c)z})$ and $\overline{u}^-(z)=a\exp({\overline{\lambda}_0^+(c)z})$ are the sub-solution and super-solution of 
\begin{equation}\label{eq:uminu}
    \left\{\begin{aligned}
    &u^-_{zz}-cu^-_z+f_0(u^-)=0, \\ & u^-(0)=a, u^-(-\infty)=0.
    \end{aligned} \right.
\end{equation}
Similarly, $\underline{u}^+(z)=1+(a-1)\exp({\underline{\lambda}_1^-(c)z})$ and $\overline{u}^+(z)=1+(a-1)\exp({\overline{\lambda}_1^-(c)z})$ are the super-solution and sub-solution of
\begin{equation}\label{eq:uplus}
    \left\{ \begin{aligned} 
    &u^+_{zz}-cu^+_z+f_1(u^+)=0, \\ & u^+(0)=a, u^+(+\infty)=1.
    \end{aligned} \right.
\end{equation}
\end{lemma}

\begin{lemma}\label{lem:wpm}
Assume {\rm (H1)-(H3)} hold. Then 
 \[\lim\limits_{z\uparrow 0}\underline{w}^-(z)=\underline{\lambda}_0^+(c)a>\lim\limits_{z\uparrow 0}\overline{w}^-(z)=\overline{\lambda}_0^+(c)a,\] and 
 \[\lim\limits_{z\downarrow 0}\underline{w}^+(z)=\underline{\lambda}_1^-(c)(a-1)>\lim\limits_{z\downarrow 0}\overline{w}^+(z)=\overline{\lambda}_1^-(c)(a-1).\] 
 In addition, 
 \[\underline{w}^-(0)=\underline{w}^+(0),\quad \mbox{for}~c=\underline{c},\]
 and
 \[\overline{w}^-(0)=\overline{w}^+(0),\quad \mbox{for}~c=\overline{c}.\]
\end{lemma}

\begin{lemma}\label{lem:cdepend}
Under the assumptions {\rm (H1)-(H3)}, we claim:
\begin{itemize}
    \item {\rm(I)} $w^-(z;c)$ is increasing with respect to $c$; $w^+(z;c)$ is decreasing with respect to $c$.
    \item {\rm(II)} If $c=0$, then $\overline{w}^-(0)<\underline{w}^-(0)<\overline{w}^+(0)<\underline{w}^+(0)$.
    \item {\rm(III)} $\underline{w}^-(0;\check{c})=\overline{w}^+(0;\check{c})$ and $\overline{w}^-(0;\hat{c})=\underline{w}^+(0;\hat{c})$.
\end{itemize}
\end{lemma}

\begin{remark}
The conclusions in Lemma \ref{lem:cdepend} and Lemma \ref{lem:wpm} imply 
$$0<\check{c}<\min\{\underline{c},\overline{c}\}\leq\max\{\underline{c},\overline{c}\}<\hat{c}.$$
\end{remark}

\begin{theorem}[Existence and uniqueness]\label{thm:exist}
Assume that {\rm(H1)-(H3)} hold. Then there exists a positive constant $c^*$, such that \eqref{eq-RDE} has a unique $C^1$ travelling wave solution $u^*(z+c^*t)$ with positive speed $c^*$, where $\check{c}\leq c^*\leq\hat{c}$.
\end{theorem}


\begin{proof}[Proof of Lemma \ref{lem:subp-sol}]
We just give the proof of $\underline{u}^-(z)$, others are followed by using the same way.

Substitute $\underline{u}^-(z)=a\exp({\underline{\lambda}_0^+(c)z})$ into the first second-order partial differential equation of \eqref{eq:uminu}, we can get 
\[\underline{u}^-_{zz}-c\underline{u}^-_z+f(\underline{u}_z^-)\geq0,\] by using the claim (1) in Remark \ref{rmk:fu}. Thus we say $\underline{u}^-(z)$ is the super-solution of \eqref{eq:uminu}.
\end{proof}

\begin{proof}[Proof of Lemma \ref{lem:wpm}]
It is obviously by using the conclusion of \eqref{eq:under-u} and \eqref{eq:over-u}.
\end{proof}

\begin{proof}[Proof of Lemma \ref{lem:cdepend}]
(I) We consider the following ODE problem
\begin{equation}\label{eq-wum}
    \left\{\begin{aligned}
    &\frac{dw^-(u;c)}{du}=c-\frac{f_0(u)}{w^-(u;c)}, \\
    & w^-(0;c)=0,
    \end{aligned} \right. \quad \mbox{for}~u\in[0,a].
\end{equation}
Derivative with respect to $c$ in \eqref{eq-wum}, we obtain
\begin{equation*}
     \frac{dw_c^-}{du}=1+\frac{f_0(u)w_c^-}{(w^-)^2},
\end{equation*}
Let \[G^-(u)=w_c^-(u;c)\exp\left({-\int_{a/2}^u\frac{f_0(s)}{(w^-)^2(s)}ds}\right). \]
Then we get
\begin{equation*}
    \frac{dG^-(u)}{du}=\exp\left({-\int_{a/2}^u\frac{f_0(s)}{(w^-)^2(s)}ds}\right)>0.
\end{equation*}
Hence, we conclude that $G^-(u)>G^-(0)=0$, this implies $w^-_c(u;c)>0$. Thus $w^-(u;c)$ is increasing with respect to $c$.

Similarly, when we consider the following ODE problem
\begin{equation}\label{eq-wup}
    \left\{\begin{aligned}
    &\frac{dw^+}{du}=c-\frac{f_1(u)}{w^+(u;c)}.\\
    & w^+(1;c)=0,
    \end{aligned} \right.\quad \mbox{for}~u\in[a,1].
\end{equation}
Derivative with respect to $c$, we have
\begin{equation*}
    \frac{dw_c^+}{du}=1+\frac{f_1(u)w_c^+}{(w^+)^2}.
\end{equation*}
Let \[G^+(u)=w_c^+(u;c)\exp\left({\int_{u}^{1-a/2}\frac{f_0(s)}{(w^+)^2(s)}ds}\right). \]
An easily calculation showed $dG^+(u)/du>0$, then $G^+(u)<G(1)=0$. Thus $w^+_c(u;c)<0$ and $w^+(u;c)$ is decreasing with respect to $c$.

\medskip
(II) If $c=0$, then $\underline{\lambda}_0^+(0)=\sqrt{-4\underline{\alpha}}/2,\overline{\lambda}_0^+(0)=\sqrt{-4\overline{\alpha}}/2,\overline{\lambda}_1^-(0)=-\sqrt{-4\overline{\beta}}/2,\underline{\lambda}_1^-(0)=-\sqrt{-4\underline{\beta}}/2$. Hence the claim (2) in Remark \ref{rmk:fu} imply the result.(see Figure \ref{fig-wc} (a))

\medskip
(III) The results are followed in \eqref{eq:check-u} and \eqref{eq:hat-u}. (see Figure \ref{fig-wc} (b) and (c))
\end{proof}

\begin{proof}[Proof of Theorem \ref{thm:exist}]
Lemma \ref{lem:subp-sol} provides for any $c\geq0$, there are $C^0$ solutions to \eqref{eq:TWsol} with $h(u)=f(u)$:
\begin{equation*}
    \left\{\begin{aligned}
    &\underline{u}^-(z)\leq u^-(z)\leq\overline{u}^-(z), \quad &\mbox{for}~ z<0, \\  &\overline{u}^+(z)\leq u^+(z)\leq\underline{u}^+(z), \quad & \mbox{for}~ z>0. 
    \end{aligned}\right.
\end{equation*}
Combining with the Lemma \ref{lem:cdepend}, there exists a $c^*$ between $\check{c}$ and $\hat{c}$, such that $u^*(x+c^*t)$ is the unique $C^1$ solution to \eqref{eq-RDE}. 
\end{proof}

\begin{figure}
    \centering
    \subfigure[$\underline{w}^-(0)<\overline{w}^+(0)$]{\includegraphics[scale=0.5]{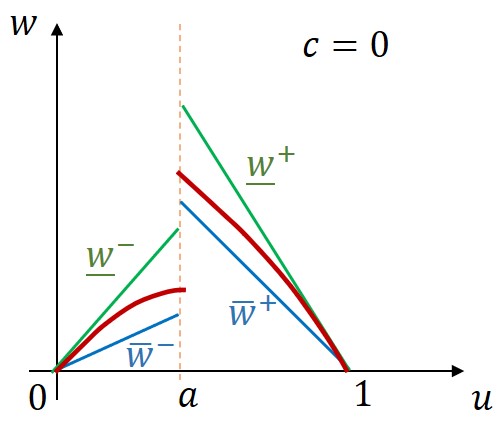}}
    \subfigure[$\underline{w}^-(0)=\overline{w}^+(0)$]{\includegraphics[scale=0.5]{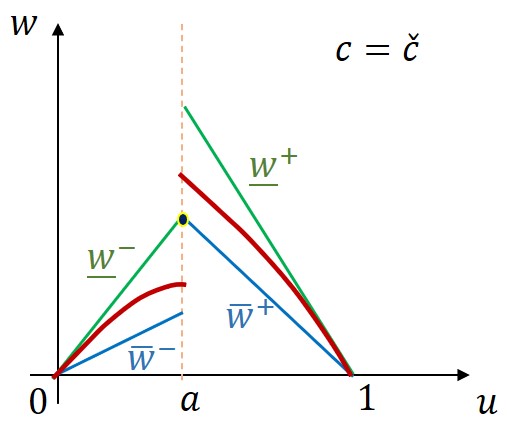}}
    \subfigure[$\overline{w}^-(0)=\underline{w}^+(0)$]{\includegraphics[scale=0.5]{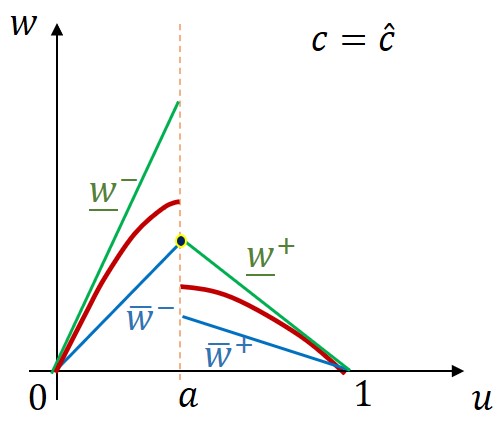}}
    \subfigure[$\overline{w}^-(0)<\overline{w}^+(0)$]{\includegraphics[scale=0.5]{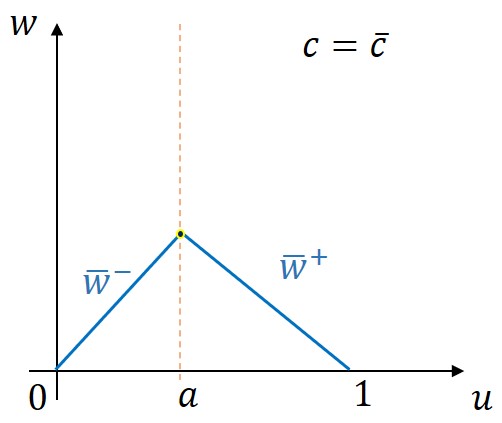}}
    \subfigure[${w}^-(0)={w}^+(0)$]{\includegraphics[scale=0.5]{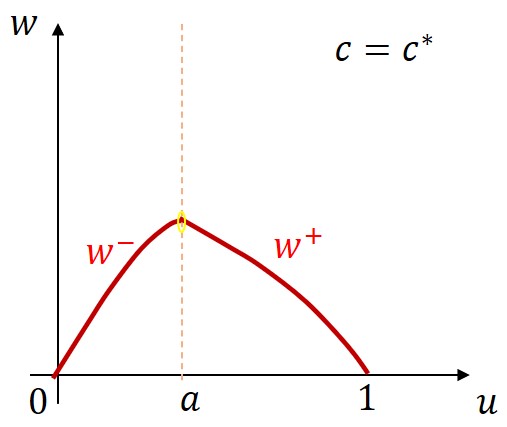}}
    \subfigure[$\underline{w}^-(0)=\underline{w}^+(0)$]{\includegraphics[scale=0.5]{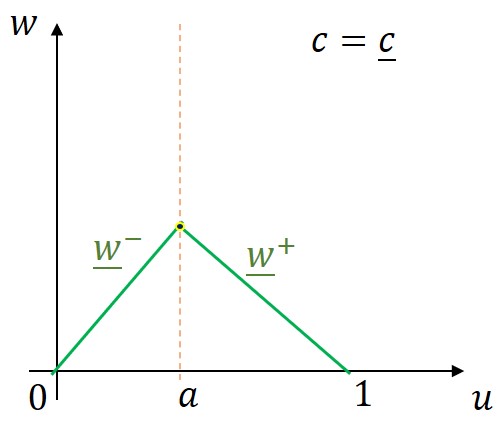}}
    \caption{$c$ is increasing from figure (a) to figure (c). On the left side of $u=a$, $\underline{w}^-,w^-,\overline{w}^-$ are moving up, on the right side of $u=a$, $\underline{w}^+,w^+,\overline{w}^+$ are moving down, as c increases. Figure (d) showed the $C^1$ solution of  $\overline{u}(z)$, figure (e) showed the $C^1$ travelling wave solution of \eqref{eq-RDE} when $c=c^*$, and figure (f) showed the $C^1$ solution of $\underline{u}(z)$. }
    \label{fig-wc}
\end{figure}

\section{Stability of travelling wave solution}

\begin{lemma}\label{lem-vphi}
Let $k$ be a bounded function, and $\varphi(x,t)$ is the solution of
\begin{equation}\label{eq:vphi}
    \left\{\begin{aligned}
    &\varphi_t-\varphi_{xx}-k\varphi\geq0, \quad & (x,t)\in\mathbb{R}\times[0,T], \\
    &\varphi(x,0)\geq0, &x\in\mathbb{R}.
    \end{aligned}
    \right.
\end{equation}
Then for every large positive constant $L$, there exists $\varepsilon(t,L)$ such that 
\begin{equation}\label{sol-vphi}
    \min\limits_{|x|<L}\varphi(x,t)\geq \varepsilon(t,L)\int_{-L}^{L}\varphi(x,t)dx. \quad \forall t\in[0,T].
\end{equation}
\end{lemma}

\begin{lemma}\label{lem:lim01}
Assume that {\rm (H1)-(H3)} hold, $u(x,t)$ is the solution of \eqref{eq-RDE} with the initial value $u_0(x)$. Then 
\begin{equation}\label{lim-1}
   \lim\limits_{t\to\infty} \mathop{\underline{\lim}}\limits_{x\to+\infty}u(x,t)=1.
\end{equation}
\begin{equation}\label{lim-0}
   \lim\limits_{t\to\infty} \mathop{\overline{\lim}}\limits_{x\to-\infty}u(x,t)=0.
\end{equation}
\end{lemma}

\begin{lemma}\label{lem:xt-MT}
Let $u(x,t)$ be the solution of \eqref{eq-RDE} and $u^*(x+ct)$ be the travelling wave solution of \eqref{eq-RDE}. Then for $\forall \varepsilon>0$, there exists $M_\varepsilon>0$ and $T_\varepsilon>0$, such that
\begin{equation}\label{eq:betw}
    u^*(x+cT_\varepsilon-M_\varepsilon)-\varepsilon\leq u(x,T_\varepsilon)\leq u^*(x+cT_\varepsilon+M_\varepsilon)+\varepsilon.
\end{equation}
\end{lemma}

\begin{lemma}\label{lem:supb}
Assume that {\rm (H1)-(H3)} hold, and $u^*(z)$ is the solution to \eqref{eq:TWsol} with $h(u)=f(u)$. Then there exist a small independent positive constant $\delta_0$ and a large positive constant $\sigma$, which depend on $u^*(z)$, such that for any $\delta\in(0,\delta_0]$ and every $z_0\in\mathbb{R}$, the functions $U^+$ and $U^-$ defined by
\begin{equation}
    U^\pm(x,t):=u^*(x+ct+z_0\pm\sigma\delta[1-e^{-\gamma t}])\pm\delta e^{-\gamma t}
\end{equation}
are a super-solution and a sub-solution to \eqref{eq-RDE} respectively. Here $\gamma:=\frac{1}{2}\min\{-f'_0(0),-f'_1(1)\}$.
\end{lemma}

\begin{theorem}[Stability]\label{Thm:stab}
Assume that {\rm(H1)-(H3)} hold, $u^*(z)$ is the travelling wave solution to \eqref{eq-RDE}. Then there exist a positive constant $\kappa$ and some positive constants $a_0, a_1$ with $0<a_0\leq a\leq a_1<1$, such that for any $u_0\in L^\infty(\mathbb{R})$ satisfying $0\leq u_0\leq1$ and \[\lim\limits_{x\to+\infty}u_0(x)>a_1,\quad \lim\limits_{x\to-\infty}u_0(x)<a_0,\]
the solution $u(x,t)$ of \eqref{eq-RDE} with initial value $u(x,0)=u_0(x)$ has the property that 
\[\|u(x,t)-u^*(x+ct+z^*)\|_{L^\infty(\mathbb{R})}\leq K e^{-\kappa t}\quad\mbox{for all }~t\geq0\]
where $z^*$ and $K$ are constants depending on $u_0(x)$.
\end{theorem}


\begin{proof}[Proof of Lemma \ref{lem-vphi}]
Since 
\[\varphi_t-\varphi_{xx}\geq k\varphi \geq-\|k\|_{\infty}\varphi,\]
where $\|k\|_\infty=\|k\|_{L^\infty}$. Hence, for $x\in[-L,L]$,
\begin{align*}
    \varphi(x,t) &\geq \frac{e^{-\|k\|_\infty t}}{2\sqrt{\pi t}} \int_{-\infty}^{\infty}e^{-\frac{(x-y)^2}{4t}}\varphi(y,0) dy \\ &\geq\frac{e^{-\|k\|_\infty t-\frac{L^2}{t}}}{2\sqrt{\pi t}}\int_{-L}^{L}\varphi(y,0) dy.  
\end{align*}
Choosing $\varepsilon(t,L)={\exp({-\|k\|_\infty t-\frac{L^2}{t}})}/{2\sqrt{\pi t}}$, thus the result is followed.
\end{proof}

\begin{proof}[Proof of Lemma \ref{lem:lim01}] 
First, let $v_0(x,t)$ be the solution of 
\begin{equation}
  \left\{ \begin{aligned} &(v_0)_{t}-(v_0)_{xx}=f_0(v_0),~(x,t)\in\mathbb{R}\times[0,\infty) \\ 
  &v_0(x,0)=u_0(x), ~\mbox{if}~ x\leq -M_0, \\ &v_0(x,0)=a_0, ~\mbox{if}~ x>-M_0.
  \end{aligned} \right.
\end{equation}
Where $M_0$ is a constant such that 
\[u_0(x)<a_0\leq a,~\mbox{for}~ \forall x\leq-M_0,\quad f_0(v)<0,~\mbox{for}~ \forall v\in(0,a].\]
Similarly, let $v_1(x,t)$ be the solution of 
\begin{equation}
  \left\{ \begin{aligned} &(v_1)_{t}-(v_1)_{xx}=f_1(v_1),~(x,t)\in\mathbb{R}\times[0,\infty) \\ 
  &v_1(x,0)=u_0(x), ~\mbox{if}~ x\geq M_1, \\ &v_1(x,0)=a_1, ~\mbox{if}~  x<M_1.
  \end{aligned} \right.
\end{equation}
Where $M_1$ is a constant such that 
\[u_0(x)>a_1\geq a, ~\mbox{for}~\forall x\geq M_1,\quad f_1(v)>0, ~\mbox{for}~ \forall v\in[a,1).\]

Second, define $q_0(t)$ and $q_1(t)$ are the solutions of 
\begin{equation}
    \left\{\begin{aligned}
    &(q_0)_t=f_0(q_0),~ t>0 \\
    & q_0(0)=a,
    \end{aligned}\right.
\end{equation}
and 
\begin{equation}
    \left\{\begin{aligned}
    &(q_1)_t=f_1(q_1),~ t>0 \\
    & q_1(0)=a,
    \end{aligned}\right.
\end{equation}
respectively. 

Since $f_1(v)>0$ for $v\in[a,1)$, $f_0(v)<0$ for $v\in(0,a]$, and $f_1(1)=0, f_0(0)=0$, we have 
\[\lim\limits_{t\to\infty}q_1(t)=1, \lim\limits_{t\to\infty}q_0(t)=0.\]
Note that $q_1(0)\leq v_1(x,0), q_0(0)\geq v_0(x,0)$, for $\forall x\in\mathbb{R}$. By comparison principle, 
\[q_1(t)\leq v_1(x,t)\leq 1, \quad \mbox{for}~\forall x\in\mathbb{R}, t>0. \]
\[q_0(t)\geq v_0(x,t)\geq 0, \quad \mbox{for}~\forall x\in\mathbb{R}, t>0. \]
Hence, 
\[\lim\limits_{t\to\infty}v_1(x,t)=1, ~\mbox{uniformly in }~x.\]
\[\lim\limits_{t\to\infty}v_0(x,t)=0, ~\mbox{uniformly in }~x.\]

In the following, we will prove \eqref{lim-1}. 
Set $w_1=u_1-v_1$, here $u_1$ is the solution of
\begin{equation*}
\left\{ \begin{aligned}
 &(u_1)_t-(u_1)_{xx}=f_1(u_1), \quad  &(x,t)\in\mathbb{R}\times[0,\infty), \\ 
 &u_1(x,0)=\max\{u_0(x),a\}, \quad &x\in\mathbb{R}. \end{aligned}\right. 
\end{equation*}
we have
\[(w_1)_t-(w_1)_{xx}=f_1(u_1)-f_1(v_1)=k(x,t)w.\]
Where $k(x,t)=\int_0^1f'_1(v_1+s(u_1-v_1))ds$.

Set $K_1=\max_{a\leq u\leq1}|f'_1(u)|$.
We define $w_1^+$ and $w_1^-$ are the solutions of
\begin{equation*}
    \left\{\begin{aligned}
    &(w_1^+)_t-(w_1^+)_{xx}=K_1w_1^+,  \\ &w_1^+(x,0)=\max\{w_1(x,0),0\}, 
    \end{aligned}\right.
\end{equation*}
and
\begin{equation*}
    \left\{\begin{aligned}
    &(w_1^-)_t-(w_1^-)_{xx}=K_1w_1^-,  \\ &w_1^-(x,0)=\min\{w_1(x,0),0\},
    \end{aligned}\right.
\end{equation*}
respectively. Then $w_1(x,0)=w_1^+(x,0)+w_1^-(x,0)$,
Hence, 
$w_1(x,t)=w_1^+(x,t)+w_1^-(x,0)$.

Since $w_1^+(x,0)\geq0$, by the Maximum principle,
\[0\leq w_1^+(x,t)\leq\int_{-\infty}^\infty \frac{\exp(K_1t-\frac{|x-y|^2}{4t})}{2\sqrt{\pi t}} w_1^+(y,0) dy.\]
Similarly,
\[0\leq -w_1^-(x,t)\leq -\int_{-\infty}^\infty \frac{\exp(K_1t-\frac{|x-y|^2}{4t})}{2\sqrt{\pi t}} w_1^-(y,0) dy.\]
Therefore,
\[|w_1(x,t)|=w_1^+(x,t)-w_1^-(x,t)\leq \int_{-\infty}^\infty \frac{\exp(K_1t-\frac{|x-y|^2}{4t})}{2\sqrt{\pi t}} |w_1(y,0)| dy.\]
Note that $w_1(x,0)=0$, when $x\geq M$; $w_1(x,0)\leq 1$, when $x<M$.
Hence,
\begin{align*}
    |w_1(x,t)| &\leq\int_{-\infty}^M \frac{\exp(K_1t-\frac{|x-y|^2}{4t})}{2\sqrt{\pi t}} dy \\ 
    &=\frac{e^{K_1t}}{\sqrt{\pi}}\int_{z}^{\infty}e^{-s^2}ds.
\end{align*}
Here $z=\frac{x-M}{2\sqrt{t}}$.
Notice that
\[\int_{z}^{\infty}e^{-s^2}ds\leq\frac{1}{2z}e^{-z^2}\leq\frac{\sqrt{t}\exp(K_1t-\frac{(x-M)^2}{4t})}{x-M},\quad\mbox{when}x\geq M.\]
Hence, when $x\geq M+\sqrt{4K}t$,
\[|u_1(x,t)-v_1(x,t)|\leq\frac{\sqrt{t}}{x-M}\leq\frac{1}{\sqrt{4K_1t}}.\]
Thus,
\[\lim_{t\to\infty} \mathop{\underline{\lim}}\limits_{x\to+\infty}u(x,t)=\lim_{t\to\infty} \mathop{\underline{\lim}}\limits_{x\to+\infty}u_1(x,t)=1.\]
In the same way, if we set $w_0=u_0-v_0$, here $u_0$ is the solution of
\begin{equation*}
\left\{ \begin{aligned}
 &(u_0)_t-(u_0)_{xx}=f_0(u_0), \quad  &(x,t)\in\mathbb{R}\times[0,\infty), \\ 
 &u_0(x,0)=\min\{u_0(x),a\}, \quad &x\in\mathbb{R}. \end{aligned}\right. 
\end{equation*}
Then we can also prove
\[\lim_{t\to\infty} \mathop{\overline{\lim}}\limits_{x\to+\infty}u(x,t)=\lim_{t\to\infty} \mathop{\overline{\lim}}\limits_{x\to-\infty}u_0(x,t)=0.\]
\end{proof}

\begin{proof}[Proof of Lemma \ref{lem:xt-MT}]
From Lemma \ref{lem:lim01}, we know for $\forall \varepsilon>0$, $\exists T_\varepsilon>0, \tilde{M}_\varepsilon>0$, such that
\[u(x,T_\varepsilon)\geq 1-\varepsilon, \quad \forall x\geq \tilde{M}_\varepsilon,\]
\[u(x,T_\varepsilon)\leq\varepsilon, \quad \forall x\leq-\tilde{M}_\varepsilon.\]
Let $a_\varepsilon$ be the point such that
\[u^*(-a_\varepsilon+cT_\varepsilon)\leq\varepsilon,\quad u^*(a_\varepsilon+cT_\varepsilon)\geq1-\varepsilon.\]
Set $M_\varepsilon=\tilde{M}_\varepsilon+a_\varepsilon$, we have
\begin{equation*}
    u^*(x+cT_\varepsilon+M_\varepsilon)+\varepsilon \geq\left\{\begin{aligned}
    & 1, &\mbox{if}~x\geq-\tilde{M}_\varepsilon, \\ &\varepsilon, &\mbox{if}~ x<-\tilde{M}_\varepsilon. 
    \end{aligned}\right.
\end{equation*}
\begin{equation*}
    u^*(x+cT_\varepsilon-M_\varepsilon)-\varepsilon\leq\left\{\begin{aligned}
    & 0, &\mbox{if}~x\leq\tilde{M}_\varepsilon, \\ &1-\varepsilon, &\mbox{if}~x>\tilde{M}_\varepsilon.
    \end{aligned}\right.
\end{equation*}
Thus we claim \eqref{eq:betw} hold.
\end{proof}

\begin{proof}[Proof of Lemma \ref{lem:supb}]
We define a large positive constant $M\gg 1$, and define 
\[K_0=\max\limits_{0\leq u\leq a}|f'_0(u)|,\quad K_1=\max\limits_{a\leq u\leq1}|f'_1(u)|,\quad K_2=\max\limits_{x\in(0,a),y\in(a,1)}\frac{f_1(y)-f_0(x)}{y-x},\] \[\delta_0=\frac{\gamma}{K_0+K_1+K_2},\quad \sigma=\frac{\gamma+K_0+K_1+K_2}{\gamma\epsilon^*},\quad \epsilon^*=\min\limits_{|z|\leq M}|u^*_z(z)|.\]
Substitute $U^+(x,t)$ into \eqref{eq-RDE}, and notice $u^*(z)$ satisfies $u''-cu'+f(u)=0$, then we derive
\begin{equation}\label{eq-supU}
    U^+_t-U^+_{xx}-f(U^+)=\delta e^{-\gamma t}\left[\sigma\gamma u_z^*-\gamma+\delta e^{\gamma t}(f(u^*)-f(u^*+\delta e^{-\gamma t}))\right].
\end{equation}

Next, we will calculate the right-hand side of \eqref{eq-supU} from $|z|>M$ and $|z|\leq M$.
For $|z|>M$. Since $u^*(z)\to 0,$ as $z\to-\infty$, $u^*(z)\to 1$, as $z\to+\infty$, also notice that $u_z^*(\pm\infty)=0$. By the definition of $K_0, K_1$. Thus 
\[f(u^*)-f(u^*+\delta e^{-\gamma t})\geq-K_0\delta e^{-\gamma t},\quad z<-M.\]
\[f(u^*)-f(u^*+\delta e^{-\gamma t})\geq-K_1\delta e^{-\gamma t},\quad z>M.\]
For $|z|\leq M$. Because of the definition of $K_2$, it holds
\[f(u^*)-f(u^*+\delta e^{-\gamma t})\geq-K_2\delta e^{-\gamma t},\quad |z|\leq M.\]
Combining the definition of $\sigma$, then
\[U^+_t-U^+_{xx}-f(U^+)\geq0.\]

In the same way to consider $U^-(x,t)$, we substitute it into \eqref{eq-RDE}, then we derive
\[U^-_t-U^-_{xx}-f(U^-)=\delta e^{-\gamma t}\left[\gamma-\sigma\gamma u_z^*+\delta e^{\gamma t}(f(u^*)-f(u^*-\delta e^{-\gamma t}))\right]\leq0.\]
Above all the statements we just prove the conclusions in the Lemma.
\end{proof}

\begin{proof}[Proof of Theorem \ref{Thm:stab}]

Applying Lemma \ref{lem:supb}, by choosing $z_0=0$ in $U^\pm(x,t)$, and $z=x+ct$. It yields
\[u^*(z-\sigma\delta[1-e^{-\gamma t}])-\delta e^{-\gamma t}\leq u(z-ct,t)\leq u^*(z+\sigma\delta[1-e^{-\gamma t}])+\delta e^{-\gamma t}.\]

For every $t>0$, set
\[\Sigma(t)=\{(Z_1(t),Z_2(t),\delta(t))\mid   u^*(z+Z_1(t))-\delta(t) \leq u(z-ct,t)\leq u^*(z+Z_2(t))+\delta(t)\}.\]
\[D(t)=\min\limits_{(Z_1,Z_2,\delta)\in\Sigma(t)}\{Z_2(t)-Z_1(t)+A\delta(t)\}.\]
Where $0\leq\eta\leq\varepsilon_0,\xi_i\leq M_{\varepsilon_0},$

We want to show that there exists positive constant $\theta<1$, for any $\tau>0$, it hold 
\[D(t+\tau)\leq \theta D(t).\]
We divide the proof into two case:

Case 1: $Z_2(t)-Z_1(t)\leq A\delta(t)$.
Then, for any $x\in\mathbb{R},c>0$,
\[u(x,t+\tau)\leq u^*(x+ct+c\tau+Z_2(t)+\sigma\delta(t)[1-e^{-\gamma\tau}])+\delta(t)e^{-\gamma\tau}\]
\[u(x,t+\tau)\geq u^*(x+ct+c\tau+Z_1(t)-\sigma\delta(t)[1-e^{-\gamma\tau}])-\delta(t)e^{-\gamma\tau}\]
Set 
\[Z_1(t+\tau)=Z_1(t)-\sigma\delta(t)[1-e^{-\gamma\tau}],\]
\[Z_2(t+\tau)=Z_2(t)+\sigma\delta(t)[1-e^{-\gamma\tau}],\]
\[\delta(t+\tau)=\delta(t)e^{-\gamma\tau}.\]
Then
\[(Z_1(t+\tau),Z_2(t+\tau),\delta(t+\tau))\in\Sigma(t+\tau).\]
Hence
\begin{align*}
    D(t+\tau)&\leq Z_2(t+\tau)-Z_1(t+\tau)+A\delta(t+\tau) \\ &=Z_2(t)-Z_1(t)+2\sigma\delta(t)[1-e^{-\gamma\tau}]+A\delta(t)e^{-\gamma\tau} \\ &\leq D(t)\left(\frac{1}{2}+\frac{2\sigma}{A}+e^{-\gamma\tau}\right) \\ &\leq \frac{3}{4}D(t)
\end{align*}
By choosing $\tau\geq\frac{\ln8}{\gamma}, A=16\sigma$.

Case 2: $Z_2(t)-Z_1(t)\geq A\delta(t)$.

Using the definition of $\epsilon^*$ in the proof of Lemma \ref{lem:supb}, it follows that
\begin{equation}\label{eq-intu}
\int_{-M}^Mu^*(z+\sigma\delta[1-e^{-\gamma t}])-u^*(z-\sigma\delta[1-e^{-\gamma t}]) dz \geq 4\epsilon^*M\sigma\delta[1-e^{-\gamma t}].
\end{equation}
We assume that 
\[\int_{-M}^{M}u^*(z+\sigma\delta[1-e^{-\gamma t}])-u(z-ct,t) dz\geq \int_{-M}^{M} u(z-ct,t)-u^*(z-\sigma\delta[1-e^{-\gamma t}]) dz.\]
Then from \eqref{eq-intu}, we have
\begin{equation}\label{eq:2intu}
\int_{-M}^{M}u^*(z+\sigma\delta[1-e^{-\gamma t}])-u(z-ct,t) dz\geq 2\epsilon^*M\sigma\delta[1-e^{-\gamma t}]. 
\end{equation}

Let 
\[\Phi(x,t)=u^*(x+ct+\sigma\delta[1-e^{-\gamma t}])+\delta e^{-\gamma t}-u(x,t),\]
then $\Phi(x,t)$ is the solution of \eqref{eq:vphi} with $k\varphi=K\Phi$, thus we applying \eqref{sol-vphi} to get
\[\Phi(x,\tau)\geq \varepsilon(\tau,L)\int_{-L}^{L}\Phi(x,\tau)dx.\]
Combing \eqref{eq:2intu}, then when $|x|\leq L$, we have
\begin{equation}
    u^*(x+c\tau+\sigma\delta[1-e^{-\gamma \tau}])+\delta e^{-\gamma \tau}-u(x,\tau)\geq\varepsilon(\tau,L) L\delta [2\epsilon^*\sigma(1-e^{-\gamma\tau})+e^{-\gamma\tau}].
\end{equation}

When $|x|>L$,
\[u(x,\tau)\leq u^*(x+c\tau-H(\tau)\delta+\sigma\delta[1-e^{-\gamma\tau}])+\delta e^{-\gamma\tau}.\]
Here $H(\tau)=2\varepsilon(\tau,L)L\epsilon^*\sigma$.
Thus, 
\[u(x,t+\tau)\leq u^*(x+ct+c\tau-H(\tau)\delta+\sigma\delta[1-e^{-\gamma \tau}])+H(\tau)\delta+\delta e^{-\gamma\tau}. \]
Similarly,
\[u(x,t+\tau)\geq u^*(x+ct+c\tau+H(\tau)\delta-\sigma\delta[1-e^{-\gamma \tau}])-H(\tau)\delta-\delta e^{-\gamma\tau}. \]

Set
\[Z_1(t+\tau)=Z_1(t)+H(\tau)\delta(t),\]
\[Z_2(t+\tau)=Z_2(t)-H(\tau)\delta(t),\]
\[\delta(t+\tau)=\delta(t)e^{-\gamma\tau}+H(\tau)\delta(t).\]
Hence,
\begin{align*}
    D(t+\tau)&\leq Z_2(t+\tau)-Z_1(t+\tau)+A\delta(t+\tau) \\ &= Z_2(t)-Z_1(t)+(A-2)H(\tau)\delta(t)+A\delta(t)e^{-\gamma\tau} \\ &\leq (Z_2(t)-Z_1(t))+\left(\frac{A-2}{A}H(\tau)+e^{-\gamma\tau}\right)A\delta(t) \\ &\leq\nu D(t), 
\end{align*}
where $0<\nu<1$.

Therefore, $\exists 0<\theta<1$, such that
\[D(t+1)\leq\theta D(t),\quad \forall t\geq 1.\]
Then, $D(t)\leq Me^{t\ln\theta}$.
Therefore, there exists $0\leq\xi(t)\leq Me^{-\rho t}, 0\leq\eta(t)\leq Me^{-\rho t}$, such that 
\[\|u(x,t)-u^*(x+ct+\xi(t))\|_{\infty} =o(1)e^{-\rho t}.\]
Here $\rho=-\ln\theta$.
\end{proof}

\section*{Acknowledgments}
This work is inspired by \cite[Theorem 3.1]{HKMT}, in which it takes a distinct method into account.


\end{document}